\documentclass[reqno]{amsart}
\usepackage{amscd,amsfonts,amssymb}
\usepackage{tikz}
\usepackage{pgfplots}
\usetikzlibrary{automata,positioning,calc,trees}
\usetikzlibrary{intersections,pgfplots.fillbetween}
\usetikzlibrary{shapes.geometric, arrows}
\usepackage{graphicx,tikz}
\usepackage{pgf,tikz,pgfplots}
\usepackage{mathrsfs}
\usepackage{enumerate}
\usepackage[shortlabels]{enumitem}
\usepackage{mathrsfs}
\usepackage{amssymb,amsmath,amsthm,color}
\usepackage{caption,subcaption}
\usepackage{hyperref}
\usepackage{cleveref}
\usepackage[alphabetic,bibtex-style]{amsrefs}
\usepackage{url}
\usepackage{setspace}
\usepackage{float}
\usepackage{makecell}
\BibSpec{article}{%
	+{}  {\PrintAuthors}                {author}
	+{,} { \textit}                     {title}
	+{.} { }                            {part}
	+{:} { \textit}                     {subtitle}
	+{,} { \PrintContributions}         {contribution}
	+{.} { \PrintPartials}              {partial}
	+{,} { }                            {journal}
	+{}  { \textbf}                     {volume}
	+{}  { \PrintDatePV}                {date}
	+{,} { \issuetext}                  {number}
	+{,} { \eprintpages}                {pages}
	+{,} { }                            {status}
	+{,} { \url}                        {url}    % <---- ADDED
	+{,} { \PrintDOI}                   {doi}
	+{,} { available at \eprint}        {eprint}
	+{}  { \parenthesize}               {language}
	+{}  { \PrintTranslation}           {translation}
	+{;} { \PrintReprint}               {reprint}
	+{.} { }                            {note}
	+{.} {}                             {transition}
}
\BibSpec{book}{%
	+{}  {\PrintAuthors}                {author}
	+{,} { \textit}                     {title}
	+{.} { }                            {part}
	+{:} { \textit}                     {subtitle}
	+{,} { \PrintContributions}         {contribution}
	+{.} { \PrintPartials}              {partial}
	+{,} { }                            {series}
	+{}  { \textbf}                     {volume}
	+{}  { \PrintDatePV}                {date}
	+{,} { \issuetext}                  {number}
	+{,} { \eprintpages}                {pages}
	+{,} { }                            {status}
	+{,} { \url}                        {url}    % <---- ADDED
	+{,} { available at \eprint}        {eprint}
	+{}  { \parenthesize}               {language}
	+{}  { \PrintTranslation}           {translation}
	+{;} { \PrintReprint}               {reprint}
	+{.} { }                            {note}
	+{.} {}                             {transition}
	+{,} { \PrintDOI}                   {doi}
}
\newcommand{\R}{{\mathbb {R}}}

\newcommand{\N}{{\mathbb N}}
\newcommand{\Z}{{\mathbb Z}}
\newcommand{\C}{{\mathbb C}}

\newcommand{\T}{{\mathbb T}}

\theoremstyle{plain} % definition 
\newtheorem{lemma}[equation]{Lemma} 
\newtheorem{proposition}[equation]{Proposition} 
\newtheorem{theorem}[equation]{Theorem}

\theoremstyle{definition}
\newtheorem{definition}[equation]{Definition} 

\theoremstyle{remark}

\numberwithin{equation}{section}

\title[] {Sparse Bounds for Discrete Maximal Functions associated with Birch-Magyar averages}
\author[Bhojak]{Ankit Bhojak}
\address{Ankit Bhojak\\
	Department of Mathematics\\
	Indian Institute of Science Education and Research Bhopal\\
	Bhopal-462066, India.}
\email{ankitb@iiserb.ac.in}

\author[Choudhary]{Surjeet Singh Choudhary}
\address{Surjeet Singh Choudhary\\
	Department of Mathematics\\
	Indian Institute of Science Education and Research Mohali\\
	Mohali-140306, India.}
\email{surjeet19@iisermohali.ac.in}

\author[Samanta]{Siddhartha Samanta}
\address{Siddhartha Samanta\\
	Department of Mathematics\\
	Indian Institute of Science Education and Research Bhopal\\
	Bhopal-462066, India.}
\email{siddhartha21@iiserb.ac.in}

\author[Shrivastava]{Saurabh Shrivastava}
\address{Saurabh Shrivastava\\
	Department of Mathematics\\
	Indian Institute of Science Education and Research Bhopal\\
	Bhopal-462066, India.}
\email{saurabhk@iiserb.ac.in}
 

\begin{document}
	\begin{abstract}
	In this article, we study discrete maximal function associated with the Birch-Magyar averages over sparse sequences. We establish sparse domination principle for such operators. As a consequence, we obtain $\ell^p$-estimates for such discrete maximal function over sparse sequences for all $p>1$. The proof of sparse bounds is based on scale-free $\ell^p-$improving estimates for the single scale Birch-Magyar averages.
	\end{abstract}
	\maketitle
	%\tableofcontents
	%%%%%%%%%%%%%%% SECTION SECTION SECTION %%%%%%%%%%%%%    SECTION %%%%%%%%%%%%%
	\section{Introduction}\label{sec:intro}
	Let $\sigma$ be the normalized Lebesgue measure on the unit sphere $\mathbb{S}^{n-1}\subset\R^n, n\geq 2$. For a continuous function $f:\R^n\to\C$, the spherical averages of $f$ is defined by   
	\[\mathcal{A}_{t}f(x)=\int_{\mathbb{S}^{n-1}}f(x-t\theta)d \sigma (\theta)\tag{1.1}, \;t>0.\label{noteq111}\] 
	For $n\geq3$, the full spherical maximal function
	\[\mathcal{A}_{\star}f(x)= \sup \limits_{t>0}|\mathcal{A}_{t}f(x)|\tag{1.2}\label{noteq112}\] was shown to be bounded on $L^p(\R^n)$ by Stein \cite{Stein1} if and only if $p>\frac{n}{n-1}$. Later,  Bourgain \cite{Bourgain1} proved that the maximal operator $\mathcal{A}_{\star}$ is bounded on $L^{p}(\R^2)$ if and only if $p>2$. Further, the lacunary version of the spherical maximal function was considered. Recall that a sequence $\{\lambda_{k}\}_{k=1}^{\infty}$ is called lacunary if there exists $c>1$ such that $\frac{\lambda_{k+1}}{\lambda_{k}}\geq c,$ for all $k\geq1.$ The lacunary spherical maximal function corresponding to a lacunary sequence $\{\lambda_{k}\}_{k=1}^{\infty}$ is defined as 
	\[\mathcal{A}_{lac}f(x)= \sup \limits_{k}|\mathcal{A}_{\lambda_{k}}f(x)|.\tag{1.3}\label{noteq113}\]
	The maximal function $\mathcal{A}_{lac}$ is known to be bounded on $L^{p}(\mathbb{R}^{n})$ for $1<p<\infty$ and $n\geq 2$ by the works of Calder\'on \cite{Calderon}, Coifman-Weiss \cite{CoifmanWeiss}, and Duoandikoetxea-Rubio de Francia \cite{DuoanRubio}.

	\subsection{Discrete analogues of spherical maximal functions} In this article, we are concerned with the study of discrete analogues of spherical averages and corresponding maximal function. For a function $f:\mathbb{Z}^{n} \rightarrow \mathbb{C}$ in $\ell^{p}(\mathbb{Z}^{n})$, $1\leq p\leq \infty$ and $\lambda \geq 1 $, the discrete analogue of the euclidean spherical average \eqref{noteq111} is defined by
	\[M_{\lambda}f(x) = \frac{1}{r(\lambda)}\sum_{\substack{y\in\mathbb{Z}^{n}\\|y| = \lambda}}f(x-y),\]
	where $r(\lambda) = |\{ y \in \mathbb{Z}^{n} : |y| = \lambda \}|$.
	Similar to the continuous case as in \eqref{noteq112}) and \eqref{noteq113}, the full discrete spherical maximal function and the lacunary discrete maximal function are defined as below 
	\[M_{\star}f(x) = \sup\limits_{ \lambda \geq 1 }|M_{\lambda}f(x)|\]
	and \[M_{lac}f(x) = \sup\limits_{ k \geq 1 }|M_{\lambda_{k}}f(x)|,\]    
	where $\{\lambda_k\}$ is a lacunary sequence.
	Magyar-Stein-Wainger \cite{magyar} proved that the maximal operator $M_\star$ is bounded on $\ell^p(\Z^n)$ in the sharp range $p>\frac{n}{n-2}$ and $n\geq 5$; subsequently Ionescu \cite{Ionescu} established restricted weak-type estimate for $M_{\star}$ at the endpoint $\frac{n}{n-2}$. Note that unlike the continuous counterpart, the discerte lacunary maximal operator $M_{lac}$ fails to be bounded on $\ell^p(\Z^n)$ for values of $p$ near $1$. More precisely, a counterexample of Zienkiewicz provides construction of lacunary sequence for which the operator $M_{lac}$ is not bounded on $\ell^p(\Z^n)$ for $p<\frac{n}{n-1}$. We refer the interested reader to \cite[section 7]{Cook&Hughes1} for the explicit counstruction of the counterexample. Hughes~\cite{Hughes_Thediscretesphericalaveragesoverafamilyofsparsesequences}  showed that $M_{lac}$ is bounded on $\ell^p(\Z^n)$ for $p>\frac{n-1}{n-2}$ and $n\geq 5$. This was further improved by Kesler-Lacey-Mena \cite{Lacey3} for the range $p>\frac{n-2}{n-3}$ and $n\geq 5$. However, the question of $\ell^p(\Z^n)$-estimates of $M_{lac}$ is still open in the range $\frac{n}{n-1}\leq p\leq\frac{n-2}{n-3}$. It is interesting to know for which sequence of radii $\{\lambda_{k}\}_{k=1}^{\infty}$, the maximal function 
	$\mathcal{M}(f)(x)=\sup\limits_{ k \geq 1 }|M_{\lambda_{k}}f(y)|$ extends to a bounded operator on $\ell^p(\Z^n)$ for all $p>1$. This question was addressed by Cook \cite{Cook2} for a class of sequences, known as sparse sequences, defined as follows.
\begin{definition}\label{Unique C-1-2}
		We call a sequence $\{\lambda_{k}\}_{k=1}^{\infty}$ sparse, if there exists $\mu_k$ such that $\lambda_{k}=\mu_{k}!$, with
		\[\lim_{k\rightarrow\infty}\frac{\log \mu_k}{\log k}=\infty.\]
\end{definition} 
   \sloppy Cook \cite{Cook2} showed that the maximal function $\mathcal{M}(f)(x)=\sup\limits_{ k \geq 1 }|M_{\lambda_{k}}f(y)|$ associated with the sparse sequence $\lambda_{k}=2^{k}!, k\in \N$, is bounded on $\ell^{p}(\mathbb{Z}^{n})$ for $p>1$ and $n\geq 5$. Later, Kesler-Lacey-Mena \cite{Lacey3} generalized this result further by showing that the operator $\mathcal{M}$ 
    is bounded on $\ell^p(\Z^n)$ for all $p>1$ and for any sparse sequence $\{\lambda_{k}\}_{k=1}^{\infty}$.
     
   \subsection{Birch-Magyar averages and associated maximal functions} The goal of this article is to extend the results mentioned above regarding the $\ell^{p}(\mathbb{Z}^{n})$-boundedness of maximal functions associated with sparse sequences in the setting of Birch-Magyar averages. These discrete averages are defined over more general algebraic hypersurfaces. This framework was developed by Birch~\cite{Birch} and Magyar~\cite{Magyar2}.  Let us first briefly recall the notion of Birch-Magyar averages. Note that a homogeneous polynomial with integral coefficients is called an integral form. Let $\mathfrak{R}(x)$ be an integral form of degree $d>1$  in $n$ variables. The Birch rank of $\mathfrak{R}$, denoted by $\mathcal{B}(\mathfrak{R}),$ is defined as
\[\mathcal{B}(\mathfrak{R})= n- \text{dim}V(\mathfrak{R}),\]
where
 \[V(\mathfrak{R})=\{z\in\mathbb{C}^{n}:\partial_{z_{1}}\mathfrak{R}(z)=\partial_{z_{2}}\mathfrak{R}(z)=\cdots=\partial_{z_{n}}\mathfrak{R}(z)=0\}.\] 
We use the following standard notation 
$$c_{\mathfrak{R}}=\frac{\mathcal{B}(\mathfrak{R})}{(d-1)2^{d-1}}~~\text{and}~~ \eta_{\mathfrak{R}}=\frac{1}{6d}\left(\frac{c_{\mathfrak{R}}}{2}-1\right).$$
Note that  $d\eta_{\mathfrak{R}}<c_{\mathfrak{R}}-2$. For a given compactly supported smooth function $\phi$ with $0\leq\phi\leq 1$, an integral form $\mathfrak{R}(x)$ of degree $d>1$ in $n$ variables is called $\phi-$regular if there exists a non-singular solution of $\mathfrak{R}(x)=1$ in the support of $\phi$ and Birch rank of $\mathfrak{R}$ is greater than $(d-1)2^{d}$. We note that $c_{\mathfrak{R}}>2$ if $\mathfrak{R}$ is $\phi-$regular.

Consider the counting function  $r_{\mathfrak{R},\phi}(\lambda)=\sum\limits_{\mathfrak{R}(x)=\lambda}\phi\left(\frac{x}{\lambda^{\frac{1}{d}}} \right)$. It is known, see~ \cites{Birch, Magyar2}, that for a $\phi-$regular form $\mathfrak{R}$ there is an infinite arithmetic progression $\Lambda$ of regular values $\lambda\in \Lambda$ satisfying $r_{\mathfrak{R},\phi}(\lambda)\approx\lambda^{\frac{n}{d}-1}$. Henceforth, for a $\phi-$regular form $\mathfrak{R}$, let $\Lambda$ denote such an infinite  arithmetic progression. We refer the readers to \cite{Birch}, \cite{Magyar2}, and \cite{Cook&Hughes1} for more details on $\phi-$regular forms. For a $\phi-$regular form $\mathfrak{R}(x)$ of degree $d>1$  in $n$ variables and for $\lambda\in \Lambda$, the Birch-Magyar average of $f$ is defined by 
 \[M_{\lambda,\phi}^{\mathfrak{R}}f(x)=\frac{1}{r_{\mathfrak{R},\phi}(\lambda)}\sum_{y:\mathfrak{R}(y)=\lambda}\phi\left(\frac{y}{\lambda^{\frac{1}{d}}} \right)f(x-y).\tag{1.4}\label{noteq119}\]
Note that the definition above encompasses the $k-$spherical averages studied in \cite{Magyar1997}, \cite{Hughes2017}, \cite{KeHughes2017}, and \cite{ACHughes2018}, see Cook and Hughes~\cite{Cook&Hughes1} for details. Magyar \cite{Magyar2} proved that the discrete maximal function associated with Birch-Magyar averages
 \[M^{\mathfrak{R},\phi}_{*}f(x)=\sup_{\lambda\in\Lambda}|M_{\lambda,\phi}^{\mathfrak{R}}f(x)|,\]
 is bounded on $\ell^p(\Z^n)$ when $p=2$. Further, the operator $M^{\mathfrak{R},\phi}_{*}$ is bounded on $\ell^p(\Z^n)$ for $p\leq\frac{n}{n-d}$. Later, Hughes~\cite[Appendix A]{Hughes} showed that the operator $M^{\mathfrak{R},\phi}_{*}$ is not bounded on $\ell^p(\Z^n)$ for  $p>\text{max}\left\{\frac{c_{\mathfrak{R}}}{c_{\mathfrak{R}}-1},\frac{2(\eta_{\mathfrak{R}}+1)}{2\eta_{\mathfrak{R}}+1}\right\}$. However,  the sharp range of $p$ for which the operator $M^{\mathfrak{R},\phi}_{*}$ is bounded on $\ell^{p}(\Z^n)$ is not known. For a lacunary sequence $\{\lambda_{k}\}_{k=1}^{\infty}=\mathbb{L}\subset \Lambda$, Cook-Hughes \cite{Cook&Hughes1} proved that the lacunary maximal function $M_{lac}^{\mathfrak{R},\phi}f(x)=\sup\limits_{\lambda_{k}\in\mathbb{L}}|M_{\lambda_{k},\phi}^{\mathfrak{R}}f(x)|$ is bounded on $\ell^{p}(\mathbb{Z}^{n})$ for $p>\frac{2c_{\mathfrak{R}}-2}{2c_{\mathfrak{R}}-3}$, provided the dimension $n$ is sufficiently large. Moreover, for the range $1< p<\frac{n}{n-1}$, they showed that there exists a lacunary sequence $\{\lambda_{k}\}_{k=1}^{\infty}$ such that the corresponding maximal operator  $M_{lac}^{\mathfrak{R},\phi}$ is not bounded on $\ell^{p}(\mathbb{Z}^{n})$. 
 
 We address the question of $\ell^p-$boundedness of the maximal function associated with Birch-Magyar averages for sparse sequences defined in \Cref{Unique C-1-2}. In particular, we prove the following result.
	\begin{theorem}\label{Unique C-1-10}
		Let $\mathfrak{R}$ be a $\phi-$regular integral form in $n$ variables of degree $d>1$ and $\{\lambda_k\}_{k\in\N}$ be a sparse sequence contained in $\Lambda$, the set of regular values. Then the maximal function $\mathcal{M}^{\mathfrak{R},\phi}f=\sup\limits_{\lambda_{k}}|M_{\lambda_{k},\phi}^{\mathfrak{R}}f(x)|$ is bounded on $\ell^{p}(\mathbb{Z}^{n})$ for all $p>1$.
	\end{theorem}
	Indeed, we shall prove sparse bounds, \Cref{Unique C-1-9}, for the operator $\mathcal{M}^{\mathfrak{R},\phi}$. \Cref{Unique C-1-10} can be obtained as a consequence of sparse bounds. Further, note that \Cref{Unique C-1-10} is a generalization of the result obtained by Kesler-Lacey-Mena \cite{Lacey3} for the spherical case. In order to state the result concerning sparse bounds for the operator $\mathcal{M}^{\mathfrak{R},\phi}$, let us discuss the preliminaries notion for sparse bounds for discrete operators. 
	\subsection{Sparse bounds for discrete operators} 
	The study of sparse bounds for discrete operators is a new topic of research. It is well known that appropriate sparse bounds for an operator imply quantitative weighted $\ell^p$-estimates and vector-valued inequalities for the operator under consideration. We refer the reader to Lacey~\cite{Lacey2} for sparse bounds for spherical maximal functions in the continuous case and to~\cites{Lacey1,Kesler} for the discrete case. 
	
	\begin{figure}[H]
		\begin{center}
			\begin{tikzpicture}[scale=0.6]
				\draw [][->,thick] (-.5,0) -- (7.5,0) node[below] {$ \frac1p$} ;
				\draw [][->,thick] (0,-.5) -- (0,7.5) node[left] {$ \frac1q$};
				\filldraw (4.55,1.4) circle (.05em) ; 
				\filldraw (4.55,4.25) circle (.05em) ;
				\filldraw (4,4.75) circle (.05em) ;
				
				\filldraw (5, 4.8) circle (.05em) ;
				\filldraw (5.5, 5.3) circle (.05em) ;
				
				\draw[blue][dotted] (0,6) -- (5.5,5.3) -- (6,0); 
				
				\filldraw[cyan, opacity=.2] (0,6) -- (4.55,1.4)   --  (4.55,4.25)  --  (4,4.75)  -- (0,6);
				
				\filldraw[blue, opacity=.1] (0,6) -- (5, 4.8) -- (6,0) -- (4.55,1.4) --  (4.55,4.25)  --  (4,4.75) -- (0,6);
				\draw (0,6) node[left] {$Z_3=(0,1)$}; 
				
				\draw (4.5,1.4) node[below,left] {$ Z_0$};  \draw (4,4.75)  node[below, left] {$ Z_2$};   \draw  (4.5,4.25)  node[left, below] { $ Z_1$};
				
				\draw  (5.0, 4.6)  node[below, right] { $ P_2=\left(\frac{n-1}{n+1},\frac{n-1}{n+1}\right)$};
				\draw (6,0) node[below] {$(1,0)$}; \draw (5.4,5.4) node[above,right] {$ \left(\frac{n}{n+1},\frac{n}{n+1}\right)$};
			\end{tikzpicture}
		\end{center}
		\caption{Sparse bounds for the discrete spherical maximal function 
			hold for points $ (\frac{1}{p}, \frac{1}{q})$ in the interior of the polygon with vertices $Z_{0} 
			$,$Z_{1} $,$Z_{2} $, and $ Z_{3}$. Sparse bounds for $\mathcal{M}$ hold for points $ (\frac{1}{p}, \frac{1}{q})$ in the interior of the triangle with vertices $(0,1) 
			$, $P_2 $, and $(1,0)$.} \label{Spherical Sparse}
	\end{figure}
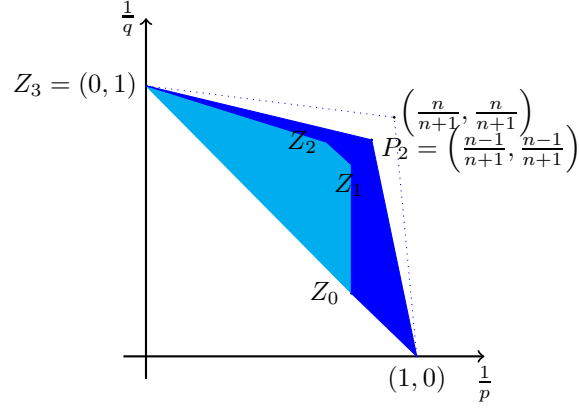

	A collection of cubes $\mathcal{S}$ in $\mathbb{Z}^{n}$ is called $\frac{1}{4}-$sparse if for each $Q \in \mathcal{S}$, there is a subset $E_{Q} \subset Q$ such that the sets $\{E_{Q} \subset Q : Q \in \mathcal{S} \}$ are pairwise disjoint and satisfy $|E_{Q}|> \frac{1}{4}|Q|$. We define the $(p, q)-$sparse form $\Lambda_{\mathcal{S},p,q}$, indexed by the sparse collection $\mathcal{S}$ as 
	\[\Lambda_{\mathcal{S},p,q}(f, g)= \sum_{Q\in \mathcal{S}}|Q|\langle f\rangle_{Q,p}\langle g\rangle_{Q,q},\]
	where $\langle f\rangle_{Q,p}=\left(\frac{1}{|Q|}\sum\limits_{m\in Q}|f(m)|^{p}\right)^{\frac{1}{p}}$ and we set $\langle f\rangle_{Q}=\langle f\rangle_{Q,1}.$ 
	
	For a sublinear operator $T$, let $\|T\|_{p,q}$ the $(p, q)-$sparse norm of $T$. It is defined as the smallest constant $C>0$ such that 
	\[|\langle Tf, g\rangle|\leq C \sup\limits_{\mathcal{S}}\Lambda_{\mathcal{S},p,q}(f,g),\]
	where $\langle\cdot,\cdot\rangle$ is the standard inner product on $\ell^2(\Z^n)$.
	
	Kesler~\cite{Kesler} proved the sparse bounds for local discrete spherical maximal function. Later Kesler-Lacey-Mena \cite{Lacey1} simplified the proof and extended the sparse bounds for the full discrete spherical maximal function $M_*$. More precisely, they proved that $\|M_{\star}\|_{p,q}<\infty$ whenever $ \left(\frac{1}{p}, \frac{1}{q}\right)$ lies in the interior of the polygon with vertices $Z_{0}=\left(\frac{n-2}{n}, \frac{2}{n}\right),\;Z_{1}=\left(\frac{n-2}{n}, \frac{n-2}{n}\right),\;Z_{2}=\left(\frac{n^{3}-4n^{2}+4n+1}{n^{3}-2n^{2}+n-2}, \frac{n^{3}-4n^{2}+6n-7}{n^{3}-2n^{2}+n-2}\right),\;\text{and}\;Z_{3}=(0,1)$, see \Cref{Spherical Sparse}. We prove the following sparse bounds for the maximal functions associated with the Birch-Magyar averages. 
	 \begin{theorem}\label{Unique C-1-9} 
		Let $\mathfrak{R}$ be a $\phi-$regular integral form of degree $d>1$  in $n$ variables and $\{\lambda_{k}\}_{k\in\N}$ be a sparse sequence contained in the set of regular values $\Lambda$. Then, the maximal function $\mathcal{M}^{\mathfrak{R},\phi}f(x)=\sup\limits_{\lambda_{k}}|M_{\lambda_{k},\phi}^{\mathfrak{R}}f(x)|$ satisfies the following sparse bounds
		\[\|\mathcal{M}^{\mathfrak{R},\phi}\|_{p,q}<\infty~\quad \text{for all}\; \left(\frac{1}{p}, \frac{1}{q}\right) \in S_n,\]
		where $S_{n}$ is the open triangle with vertices $S_1=(0,1)$, $S_2=\left(\frac{1+2\eta_{\mathfrak{R}}}{2(1+\eta_{\mathfrak{R}})},\frac{1+2\eta_{\mathfrak{R}}}{2(1+\eta_{\mathfrak{R}})}\right)$, and $S_3=(1,0)$.
	\end{theorem}

	The proof of \Cref{Unique C-1-9} relies on the method of High-Low analysis of the operator. We exploit the ideas from  Kesler-Lacey-Mena~\cite{Lacey1} along with the method of estimating the error terms from  \cites{Cook&Hughes1,Magyar2}. In order to complete our proof, we require the sparse bounds of the continuous analogue of the discrete maximal operator $\mathcal{M}^{\mathfrak{R},\phi}$. Since, the sparse bounds for the continuous analogue are not written explicitly in the literature, we provide the required details in the~\Cref{Appendix} for the convenience of the reader. Further, we require scale-free sparse bounds for the single averaging operator $M_{\lambda,\phi}^{\mathfrak{R}}$ for all $\lambda\in\Lambda$. Such estimates are of independent significance. The following result concerns the scale-free bounds for the Birch-Magyar averages. 
\begin{theorem}\label{Unique C-4-1}
    Let $\mathfrak{R}$ be a $\phi-$regular form of degree $d>1$  in $n $ variables and $S_{n}$ be the open triangle with vertices $(0,1)$, $(1,0)$, and $\left(\frac{1+2\eta_{\mathfrak{R}}}{2(1+\eta_{\mathfrak{R}})},\frac{1+2\eta_{\mathfrak{R}}}{2(1+\eta_{\mathfrak{R}})}\right)$. For $\left(\frac{1}{p},\frac{1}{q}\right)\in S_{n}$, there exists a constant $C_{p,\mathfrak{R}}$, independent of $\lambda\in\Lambda$, such that the following inequality holds 
    \[\langle M_{\lambda,\phi}^{\mathfrak{R}}f,g\rangle\leq C_{p,\mathfrak{R}}|E|\langle f\rangle_{E,p}\langle g\rangle_{E,q}.\]
    In particular, for $\left(\frac{1}{p},\frac{1}{q}\right)\in S_{n}$, we have
	\[\|M_{\lambda,\phi}^{\mathfrak{R}}\|_{\ell^p(\Z^n)\to \ell^{q'}(\Z^n)}\lesssim \lambda^{n\left(\frac{1}{q'}-\frac{1}{p}\right)}.\]
\end{theorem}
The $\ell^p$-improving estimates for single scale spherical averages were proved by \cites{Hughes,Lacey4}. They used refined Kloosterman sum estimates to improve the range of boundedness. However, we would like to remark here that such refined estimates are not available for the Birch-Magyar setting. An improvement for the Kloosterman sum in the Birch-Magyar case would lead to a better range of indices than those obtained in \Cref{Unique C-4-1}. Nonetheless, using the estimates in \cite{Hughes}, we can improve the sparse bounds obtained for the maximal operator in \Cref{Unique C-1-9} in the spherical case. More precisely, we have the following result. 
	\begin{theorem}\label{Unique C-1-8}
		Let $n\geq5$ and $P_{n}$ be the open triangle with vertices $P_1=(0,1)$, $P_2=\left(\frac{n-1}{n+1},\frac{n-1}{n+1}\right)$, and $P_3=(1,0)$. Then for all $ \left(\frac{1}{p}, \frac{1}{q}\right) \in P_{n}$, we have the sparse bound $\|\mathcal{M}\|_{p,q}<\infty.$ 
	\end{theorem}
	\subsection{Organization of the paper} The proofs of \Cref{Unique C-1-9} and \Cref{Unique C-1-8} are given in \Cref{sec:maximal}. \Cref{sec:single} contains the proof of \Cref{Unique C-4-1}. \Cref{Appendix} is devoted to proving sparse bounds for the continuous analogue of the operator $\mathcal{M}^{\mathfrak{R},\phi}$.
	
	\subsection{Notations}
	We use the following notation throughout the paper.
	\begin{enumerate}
		\item We write $1_{A}$ to denote the characteristic function of the set $A.$
		\item For a ring $K$, we use the notation $a\cdot b$ (for vectors $a,b\in K^{n}$) to denote the sum $\sum\limits_{1\leq j\leq n}a_{j}b_{j}$.
		\item The notation $\star$ denotes convolution on any group $G$, provided it is well defined.
		\item The torus $(\mathbb{R}/\mathbb{Z})^{n}$ identified with the cube $[-\frac{1}{2},\frac{1}{2}]^{n}$ and it is  denoted by $\mathbb{T}^{n}$.
		\item For a function $f$ defined on $\mathbb{Z}^{n}(\text{ or }\mathbb{R}^{n})$, we denote its $\mathbb{Z}^{n}(\text{ or }\mathbb{R}^{n})$-Fourier transform by $\widehat{f}(\xi)$ for $\xi\in\mathbb{T}^{n}(\text{ or }\mathbb{R}^{n})$. 
		\item The notation $M\lesssim_{a} N$ or $N\gtrsim_{a} M$ means that there exists an absolute constant $0<C<\infty$ depending on $a$, such that $M\leq CN.$ If there is nothing in the subscript of $\lesssim$, then $C$ is independent of the parameters on which $M$ and $N$ depend.
		\item $M\approx_{a} N$ denotes $M\lesssim_{a} N$ and $M\gtrsim_{a} N$.
		\item We denote $e(x)=e^{2\pi ix}$.
		\item The group $\mathbb{Z}_{q}$ denotes $\mathbb{Z}/q\mathbb{Z}$, and $U_{q}$ denotes $\mathbb{Z}_{q}^{\times}$ with the understanding that $U_{1}=\mathbb{Z}_{1}=\{0\}$, also denote $\Z_{q}^{n}=\Z_{q}\times\Z_{q}\times\dots\times\Z_{q}~(\text{n-times})$.
		\item For a function $f$ defined on $\R^n$, and for any $a\in\R$ we denote $f_{a}(x)$ by $\frac{1}{a^n}f\left(\frac{x}{a}\right)$, thus $\widehat{f_a}(\xi)=\widehat{f}(a\xi).$
	\end{enumerate}

	\section[Sparse bounds for the single scale Birch-Magyar average]{Sparse bounds for the single scale Birch-Magyar average: Proof of \Cref{Unique C-4-1}}\label{sec:single}
First, note that using standard arguments, it is enough to prove the theorem for characteristic functions. Let $E=\left[0,\lambda^{\frac{1}{d}}\right]^n\cap\Z^n$ and consider $f=1_F $ and $ g=1_G$, where $F,G\subseteq E$.
It is elementary to verify that  
\begin{align*}
    \|M_{\lambda,\phi}^{\mathfrak{R}}f\|_{\ell^{\infty}}&=\sup\limits_{x}\left|\frac{1}{\lambda^{\frac{n}{d}-1}}\sum_{\mathfrak{R}(y)=\lambda}f(x-y)\right|\\&\lesssim\|f\|_{\ell^\infty},
\end{align*}
and 
\begin{align*}
    \|M_{\lambda,\phi}^{\mathfrak{R}}f\|_{\ell^{1}}&\leq \frac{1}{\lambda^{\frac{n}{d}-1}}\sum_{x\in\Z^n}\sum_{\mathfrak{R}(y)=\lambda}|f(x-y)|\\&\lesssim\|f\|_{\ell^1}.
\end{align*}
Therefore, the operator $M_{\lambda,\phi}^{\mathfrak{R}}$ is bounded from $\ell^p$ to $\ell^{q'}$ for $\left(\frac{1}{p},\frac{1}{q}\right)=(0,1)$ and $\left(\frac{1}{p},\frac{1}{q}\right)=(1,0)$. Using the interpolation arguments, observe that it is enough to prove that the operator $M_{\lambda,\phi}^{\mathfrak{R}}$ is bounded from $\ell^p$ to $\ell^{q'}$ at the point $\left(\frac{1}{p},\frac{1}{q}\right)=\left(\frac{1+2\eta_{\mathfrak{R}}}{2(1+\eta_{\mathfrak{R}})},\frac{1+2\eta_{\mathfrak{R}}}{2(1+\eta_{\mathfrak{R}})}\right)$. 

We shall decompose $M_{\lambda,\phi}^{\mathfrak{R}}f$ into several but finitely many terms and prove appropriate estimate on each term. More precisely, for every term, say $\mathcal Tf$, in the decomposition of $M_{\lambda,\phi}^{\mathfrak{R}}f$, and for any integer $N$ we will establish one of the following two type of estimates. 
\begin{eqnarray}\label{equation4.1}
	\langle\mathcal{T}f\rangle_{E,\infty}\lesssim N^d\langle f\rangle_{E}
	\end{eqnarray}
\begin{eqnarray}\label{equation4.2}\langle\mathcal{T}f\rangle_{E,2}\lesssim N^{-d\eta_{\mathfrak{R}}}\langle f\rangle_{E}^{\frac{1}{2}}
	\end{eqnarray}
Note that combining these estimates together, we get that 
\[|E|^{-1}\langle M_{\lambda,\phi}^{\mathfrak{R}}f,g\rangle\lesssim N^d\langle f\rangle_{E}\langle g\rangle_{E}+N^{-d\eta_{\mathfrak{R}}}\langle f\rangle_{E}^{\frac{1}{2}}\langle f\rangle_{E}^{\frac{1}{2}}.\] 
Then optimizing over $N$ we obtain 
\[|E|^{-1}\langle M_{\lambda,\phi}^{\mathfrak{R}}f,g\rangle\lesssim[\langle f\rangle_{E}\langle g\rangle_{E}]^{\frac{1+2\eta_{\mathfrak{R}}}{2(1+\eta_{\mathfrak{R}})}}.\] Therefore, we need to prove either \eqref{equation4.1} or \eqref{equation4.2} for each part in the decomposition of the operator. 

Let $N$ be a positive integer and decompose the operator as follows 
\[M_{\lambda,\phi}^{\mathfrak{R}}f=1_{\{\lambda^{\frac{1}{d}}\leq N\}}M_{\lambda,\phi}^{\mathfrak{R}}f+1_{\{\lambda^{\frac{1}{d}}> N\}}M_{\lambda,\phi}^{\mathfrak{R}}f.\]
We set $M_{1,1}f:=1_{\{\lambda^{\frac{1}{d}}\leq N\}}M_{\lambda,\phi}^{\mathfrak{R}}f$. For $\lambda^{\frac{1}{d}}\geq N$, by following the approach of Magyar \cite[Lemma 1]{Magyar2}, we decompose the operator further as 
$M_{\lambda,\phi}^{\mathfrak{R}}=C_{\lambda}+M_{2,1}$ , where the multiplier $m_{2,1}$ of the error term given by the convolution operator $M_{2,1}$ satisfies
\begin{eqnarray}\label{equation4.3}
	\sup\limits_{\xi}|m_{2,1}(\xi)|\lesssim_{\eta_{\mathfrak{R}}}\lambda^{-\eta_{\mathfrak{R}}},
	\end{eqnarray}
and the multiplier corresponding to the operator $C_{\lambda}$ is given by
\[c_{\lambda}(\xi)=\sum_{1\leq q\leq \lambda^{\frac{1}{d}}}\sum_{b\in\mathbb{Z}^{n}_{q}}\sum_{a\in\mathbb{Z}_{q}^{\times}}e_{q}(-a\lambda)F_{q}(a,b)\widehat{\zeta_{q}}\left(\xi-\frac{b}{q}\right)\widehat{d\sigma_{\mathfrak{R}}}\left(\lambda^{\frac{1}{d}}\left(\xi-\frac{b}{q}\right)\right).\]
Here in the expression above, $\zeta$ is a radial Schwartz function on $\mathbb{R}^{n}$ such that $\widehat{\zeta}(\xi)$ is identically $1$ for $|\xi|\leq\frac{1}{2}$ and $0$ if $|\xi|\geq 1$ with $\widehat{\zeta}_{q}(\xi)=\widehat{\zeta}(q\xi)$. The measure $d\sigma_{\mathfrak{R}}$ is defined by \[d\sigma_{\mathfrak{R}}(y)=\phi(y)\frac{d\mu(y)}{|\nabla\mathfrak{R}(y)|},\] where $d\mu$ is euclidean surface measure on the hypersurface $\{x\in\mathbb{R}^{n}:\mathfrak{R}(x)=1\}$.

Let $F_{q}(a,b)$ denote the normalized Weyl sum defined as  \[F_{q}(a,b)=\frac{1}{q^{n}}\sum_{m\in\mathbb{Z}_{q}^{n}}e\left(\mathfrak{R}(m)\frac{a}{q}+\frac{m\cdot b}{q} \right),\]  where $a\in\mathbb{Z}_{q}$ and $b\in\mathbb{Z}_{q}^{n}$. 
We have the following estimate~\cite[p. 3862]{Cook&Hughes1} for the quantity $F_{q}(a,b)$: 
\[|F_{q}(a,b)|\lesssim q^{-c_{\mathfrak{R}}},\quad\text{for}\quad (a,q)=1.\]
 
We will require to decompose the operator $C_\lambda$ further as follows,
\[C_\lambda=M_{1,2}+M_{2,2}+M_{2,3},\]
where the operators $M_{1,2}$, $M_{2,2}$, and $M_{2,3}$ correspond to the multipliers $m_{1,2}$, $m_{2,2}$, and $m_{2,3}$ respectively. These are given by
\begin{align*}
	m_{1,2}(\xi)&=\sum_{1\leq q\leq N}\sum_{b\in\mathbb{Z}^{n}_{q}}\sum_{a\in\mathbb{Z}_{q}^{\times}}e_{q}(-a\lambda)F_{q}(a,b)\widehat{\zeta_{\frac{q\lambda^{\frac{1}{d}}}{N}}}\left(\xi-\frac{b}{q}\right)\widehat{d\sigma_{\mathfrak{R}}}\left(\lambda^{\frac{1}{d}}\left(\xi-\frac{b}{q}\right)\right),\\
	m_{2,2}(\xi)&=\sum_{1\leq q\leq N}\sum_{b\in\mathbb{Z}^{n}_{q}}\sum_{a\in\mathbb{Z}_{q}^{\times}}e_{q}(-a\lambda)F_{q}(a,b)\left(\widehat{\zeta_{q}}\left(\xi-\frac{b}{q}\right)-\widehat{\zeta_{\frac{q\lambda^{\frac{1}{d}}}{N}}}\left(\xi-\frac{b}{q}\right)\right)\widehat{d\sigma_{\mathfrak{R}}}\left(\lambda^{\frac{1}{d}}\left(\xi-\frac{b}{q}\right)\right),\\
	m_{2,3}(\xi)&=\sum_{N<q\leq\lambda^{\frac{1}{d}}}\sum_{b\in\mathbb{Z}^{n}_{q}}\sum_{a\in\mathbb{Z}_{q}^{\times}}e_{q}(-a\lambda)F_{q}(a,b)\widehat{\zeta_{q}}\left(\xi-\frac{b}{q}\right)\widehat{d\sigma_{\mathfrak{R}}}\left(\lambda^{\frac{1}{d}}\left(\xi-\frac{b}{q}\right)\right).
\end{align*}
Next, we prove estimates \eqref{equation4.1} or \eqref{equation4.2} for each of the terms as above. 
\subsection{\textbf{Estimate for $M_{1,1}$}} Since $r_{\mathfrak{R},\phi}(\lambda)\approx\lambda^{\frac{n}{d}-1}$ and $\lambda^{\frac{1}{d}}\leq N$, we have
\begin{align*}
	\langle M_{\lambda,\phi}^{\mathfrak{R}}f\rangle_{E,\infty}&=\frac{1}{r_{\mathfrak{R},\phi}(\lambda)}\sum_{x:\mathfrak{R}(x)=\lambda}\phi(x/\lambda^{1/d})f(y-x)\\
	&\lesssim\frac{N^d}{\lambda^{\frac{n}{d}}}\|f\|_{\ell^1}\\
	&\lesssim N^d\langle f\rangle_{E}.
\end{align*}
\subsection{\textbf{Estimate for $M_{1,2}$}} Consider the kernel $K_{1,2}(x)=(m_{1,2})^{\vee}(x)$ of the operator $M_{1,2}$. We have
\begin{align*}
	&|K_{1,2}(x)|\\
   	=&\left|\sum_{1\leq q\leq N}\sum_{a\in\Z_{q}^{\times}}e_{q}(-a\lambda)\sum_{b\in\Z_{q}^n}F_q(a,b)\int_{\T^n}e(x\cdot\xi)\widehat{\zeta_{\frac{q}{N}}}\left(\lambda^{\frac{1}{d}}\left(\xi-\frac{b}{q}\right)\right)\widehat{d\sigma_{\mathfrak{R}}}\left(\lambda^{\frac{1}{d}}\left(\xi-\frac{b}{q}\right)\right)d\xi\right|\\
   	=&\left|\sum_{1\leq q\leq N}\frac{1}{\lambda^{\frac{n}{d}}}\zeta_{\frac{q}{N}}\star d\sigma_{\mathfrak{R}}\left(\frac{x}{N}\right)\sum_{a\in\Z_{q}^{\times}}e_{q}(-a\lambda)\sum_{b\in\Z_{q}^n}F_q(a,b)e\left(x\cdot\frac{b}{q}\right)\right|\\
	=&\left|\frac{1}{\lambda^{\frac{n}{d}}}\sum_{1\leq q\leq N}\zeta_{\frac{q}{N}}\star d\sigma_{\mathfrak{R}}\left(\frac{x}{N}\right)\sum_{a\in\Z_{q}^{\times}}e_{q}(-a\lambda)e\left(\frac{a\mathfrak{R}(x)}{q}\right)\right|\\
	\lesssim&\frac{1}{\lambda^{\frac{n}{d}}}\sum_{1\leq q\leq N}q\left(\frac{N}{q}\right)\frac{1}{\left(1+\left|\frac{x}{\lambda^{\frac{1}{d}}}\right|\right)^{2n}}\\
	\lesssim&\frac{1}{N^{n-2}},
\end{align*}
    where we have used the following decay estimate (see \cite[p. 970]{Hughes}) in the second last step,
	\[|\zeta_{t}\star d\sigma_{\mathfrak{R}}(x)|\lesssim t^{-1}\frac{1}{\left(1+\left|\frac{x}{t}\right|\right)^{2n}}.\]
Hence, $\|M_{1,2}f\|_{\ell^{\infty}}\lesssim\|K_{1,2}(x)\|_{\ell^{\infty}}\|f\|_{\ell^1}\lesssim\frac{1}{N^{n-2}}\|f\|_{\ell^1}.$ This yields that \[\langle M_{1,2}f\rangle_{E,\infty}\lesssim N^2\langle f\rangle_{E}.\]

\subsection{\textbf{Estimate for $M_{2,1}$}}This follows by applying Plancherel's theorem. Recall  the inequality \eqref{equation4.3} to get that
\[\|M_{2,1}f\|_{\ell^{2}}\lesssim N^{-d\eta_{\mathfrak{R}}}\|f\|_{\ell^2}.\]
Thus, we have
\[\langle M_{2,1}f\rangle_{E,2}\lesssim N^{-d\eta_{\mathfrak{R}}}\langle f\rangle_{E}^{\frac{1}{2}}.\]
\subsection{\textbf{Estimate for $M_{2,2}$}} Note that we have the following decay estimate from \cite[Lemma 6]{Magyar2},
\[|\widehat{d\sigma}_{\mathfrak{R}}(\xi)|\lesssim\frac{1}{(1+|\xi|)^{c_{\mathfrak{R}}-1}}.\]
Further, since  $\widehat{\zeta_{q}}\left(\xi-\frac{b}{q}\right)-\widehat{\zeta_{\frac{q\lambda^{\frac{1}{d}}}{N}}}\left(\xi-\frac{b}{q}\right)$ is supported in the annulus $\left\{\xi\in\R^n:\;\frac{N}{q\lambda^{\frac{1}{d}}}\lesssim\left|\xi-\frac{b}{q}\right|\lesssim\frac{1}{q}\right\}$, we get that  
\begin{align*}
    |m_{2,2}(\xi)|&\lesssim\sum_{1\leq q\leq N}\sum_{a\in\mathbb{Z}_{q}^{\times}}\frac{1}{q^{c_{\mathfrak{R}}}}\frac{1}{\left(\lambda^{\frac{1}{d}}|\xi-\frac{b}{q}|\right)^{c_{\mathfrak{R}}-1}}\\&\lesssim\sum_{1\leq q\leq N}\sum_{a\in\mathbb{Z}_{q}^{\times}}\frac{1}{q^{c_{\mathfrak{R}}}}\left(\frac{q}{N}\right)^{c_{\mathfrak{R}}-1}\\&\lesssim\frac{1}{N^{c_{\mathfrak{R}}-2}}\lesssim \frac{1}{N^{d\eta_{\mathfrak{R}}}}.
\end{align*}
Hence it follows that
\[\langle M_{2,2}f\rangle_{E,2}\lesssim N^{-d\eta_{\mathfrak{R}}}\langle f\rangle_{E}^{\frac{1}{2}}.\]
\subsection{\textbf{Estimate for $M_{2,3}$}}By Plancherel's theorem and the bound $|F_q(a,b)|\lesssim q^{-c_{\mathfrak{R}}}$, we obtain
\begin{align*}
	\|M_{2,3}f\|_{\ell^2}&\lesssim\sum_{N<q\leq\lambda^{\frac{1}{d}}}\sum_{a\in\mathbb{Z}_{q}^{\times}} q^{-c_{\mathfrak{R}}}\|f\|_{\ell^2}\\
	&\lesssim\sum_{N<q}q^{-(c_{\mathfrak{R}}-1)}\|f\|_{\ell^2}\lesssim N^{-(c_{\mathfrak{R}}-1)}\|f\|_{\ell^2}.
\end{align*}
This implies that
\[\langle M_{2,3}f\rangle_{E,2}\lesssim N^{-d\eta_{\mathfrak{R}}}\langle f\rangle_{E}^{\frac{1}{2}}.\]
This concludes the proof of \Cref{Unique C-4-1}.\qed

\section{Sparse bounds for highly composite case: Proof of \Cref{Unique C-1-9} and \Cref{Unique C-1-8}:}\label{sec:maximal}
\begin{proof}[\textbf{Proof of \Cref{Unique C-1-9}:}]
	We begin with the standard recursion argument for sparse bounds.  Let  $Q$ be a fixed dyadic cube and $f=1_{F}$ and $g=1_{G}$ be such that the sets $F$ and $G$ are contained in $3Q$ and $Q$ respectively. Let $\mathcal{Q}_{Q}$ be the collection of maximal dyadic subcubes of $Q$ satisfying 
	\[\langle f\rangle_{3P}>C\langle f\rangle_{3Q}\] 
	for a large constant $C$. 
	
	Recall the definition of admissible stopping time for a given dyadic cube. For a dyadic cube $Q$ and $f=1_{F}$ with support inside $Q$, a map $\tau:Q\rightarrow\{\lambda_k\}$ is called an admissible stopping time if for any subcube $P\subset Q$ and some large constant $C$ with $\langle f\rangle_{P}>C\langle f\rangle_{Q}$ we have  $\min\limits_{x\in P}\tau(x)>\ell(P).$

 Observe that, for a suitable choice of constant $C$, we have $$\sum_{P\in \mathcal{Q}_{Q}}|P|\leq \frac{1}{4}|Q|,$$ and for an appropriate choice of admissible stopping time $\tau$, we get that 
	
\begin{eqnarray}\label{recursive} \left\langle\sup\limits_{\lambda_{k}\leq \ell(Q)}M_{\lambda_{k},\phi}^{\mathfrak{R}}f, g\right\rangle\leq \left\langle M_{\tau,\phi}^{\mathfrak{R}}f, g\right\rangle + \sum_{P\in \mathcal{Q}_{Q}}\left\langle\sup\limits_{\lambda_{k}\leq \ell(P)}M_{\lambda_{k},\phi}^{\mathfrak{R}}(f1_{3P}), (g1_{P})\right\rangle.
\end{eqnarray}
Therefore, the desired sparse bounds follow from the following estimate  
\begin{eqnarray}\label{equation4.4}\frac{1}{|Q|}\langle M_{\tau,\phi}^{\mathfrak{R}}f, g\rangle\lesssim\langle f\rangle_{Q}^{\frac{1}{p}}\langle g\rangle_{Q}^{\frac{1}{q}},
	\end{eqnarray}
along with a recursive argument on the second term in the equation~\eqref{recursive}. This will prove the theorem for characteristic functions. Thereafter, we can conclude the result for general class of functions using the interpolation of sparse bounds as in \cite[Lemma 4.1]{Lacey1}.
Therefore, in order to prove \Cref{Unique C-1-9}, we need to establish the estimate \eqref{equation4.4}
%\begin{eqnarray}\label{equation4.4}\frac{1}{|Q|}\langle M_{\tau,\phi}^{\mathfrak{R}}f,g\rangle\lesssim\langle %f\rangle_{Q}^{\frac{1}{p}}\langle g\rangle_{Q}^{\frac{1}{q}},
%	\end{eqnarray} 
for a suitable stopping time $\tau$ with a large cube $Q$ such that $f=1_{F}$ is supported on $3Q$ and $g=1_{G}$ is supported on $Q$ and $ \left(\frac{1}{p}, \frac{1}{q}\right) \in S_n$. In this proof, we use a continuous sparse bound estimate that we prove in the last section.

 Let us fix a point $\left(\frac{1}{p}, \frac{1}{q}\right)\in S_n$. From the definition of $\mu_l$ it is clear that $\mu_l\approx l^{\alpha(l)}$ where $\alpha(l)\rightarrow \infty$ as $l\rightarrow \infty$. Therefore, for a constant $N_{\circ}>0$ we have that $\mu_{N^{\epsilon}}>N^d$ for $N>N_{\circ}$ and $d>1$. We claim that for sufficiently small $\epsilon\in (0,1)$, stopping time $\tau$ taking values in $\left\{\lambda_k\in\Lambda:\lambda_k \text{ is a sparse sequence}\right\}$, and a large natural number $N$, we can decompose the operator as 
 $$M_{\tau,\phi}^{\mathfrak{R}}f\leq \mathcal{T}_{1}f+\mathcal{T}_{2}f$$ such that 
\begin{eqnarray} \label{equation4.6}\langle \mathcal{T}_{1}f,g\rangle \lesssim N^{\epsilon}|Q|\langle f\rangle_{Q}^{\frac{1}{p}}\langle g\rangle_{Q}^{\frac{1}{q}},
	\end{eqnarray}
\begin{eqnarray} \label{equation4.7}\langle \mathcal{T}_{2}f,g\rangle \lesssim N^{-c_{\mathfrak{R}}+2}|Q|\langle f\rangle_{Q}^{\frac{1}{2}}\langle g\rangle_{Q}^{\frac{1}{2}}.
	\end{eqnarray}
Note that the constants in the inequalities above depend on $p$, $q$, and $\epsilon$. Likewise the previous section, estimates \eqref{equation4.6} and \eqref{equation4.7} imply that 
\begin{eqnarray} \label{equation4.8}|Q|^{-1}\langle M_{\tau,\phi}^{\mathfrak{R}}f, g\rangle\lesssim N^{\epsilon}\langle f\rangle_{Q}^{\frac{1}{p}}\langle g\rangle_{Q}^{\frac{1}{q}}+N^{-c_{\mathfrak{R}}+2}\langle f\rangle_{Q}^{\frac{1}{2}}\langle g\rangle_{Q}^{\frac{1}{2}}.
\end{eqnarray}
Observe that the choice $N=\langle f\rangle_{Q}^{\frac{p-2}{2p(c_{\mathfrak{R}}-2+\epsilon)}}\langle g\rangle_{Q}^{\frac{q-2}{2q(c_{\mathfrak{R}}-2+\epsilon)}}$ optimize the right hand side of \eqref{equation4.8}. Therefore, substituting the value of $N$ in \eqref{equation4.8} and taking $\epsilon$ very small, we get the desired estimate  \[|Q|^{-1}\langle M_{\tau,\phi}^{\mathfrak{R}}f, g\rangle\lesssim \langle f\rangle_{Q}^{\frac{1}{p}}\langle g\rangle_{Q}^{\frac{1}{q}}.\]
\subsection{Decomposition of $M_{\lambda_k,\phi}^{\mathfrak{R}}$:} For each $k\in\N$, we decompose the average in the following way
\begin{align*}
    M_{\tau(\cdot),\phi}^{\mathfrak{R}}f(x)&= 1_{\{\tau(\cdot)\leq\lambda_{[N^{\epsilon}]}\}}(x)M_{\tau(\cdot),\phi}^{\mathfrak{R}}f(x)+1_{\{\tau(\cdot)>\lambda_{[N^{\epsilon}]}\}}(x)M_{\tau(\cdot),\phi}^{\mathfrak{R}}f(x).
\end{align*}
We set $M_{1,1}f=1_{\{\tau(\cdot)\leq\lambda_{[N^{\epsilon}]}\}}(x)M_{\tau(\cdot),\phi}^{\mathfrak{R}}f(x)$. For a fixed $\lambda_k>\lambda_{[N^\epsilon]}$, let $M^{k}_{1,2}$ be the operator associated with the multiplier $$\Omega_{\lambda_k}(\xi)=\sum_{0\leq a<L}\sum_{b\in \mathbb{Z}^{n}_{L}}F_{L}(a,b)\widehat{\zeta_{2L}}\left(\xi-\frac{b}{L}\right)\widehat{d\sigma_{\mathfrak{R}}}\left(\lambda_k^{\frac{1}{d}}\left(\xi-\frac{b}{L}\right)\right),~~\text{where}~~L=N!.$$ 
%where $L=N!$. \\
We decompose the second term further as 
\begin{align*}
	&1_{\{\tau(\cdot)>\lambda_{[N^{\epsilon}]}\}}(x)M_{\tau(\cdot),\phi}^{\mathfrak{R}}f(x)\\
	\leq& \sup_{\lambda_k>\lambda_{[N^\epsilon]}}|M^{k}_{1,2}f(x)|+\sup_{\lambda_k>\lambda_{[N^\epsilon]}}|M^{k}_{2,1}f(x)|+\sup_{\lambda_k>\lambda_{[N^\epsilon]}}|M^{k}_{2,2}f(x)|+\sup_{\lambda_k>\lambda_{[N^\epsilon]}}|M^{k}_{2,3}f(x)|\\
	:=&M_{1,2}f(x)+M_{2,1}f(x)+M_{2,2}f(x)+M_{2,3}f(x),
\end{align*}
where the operators $M^{k}_{2,1}$, $M^{k}_{2,2}$ and $M^{k}_{2,3}$ are associated with the multipliers $m^{k}_{2,1}$, $m^{k}_{2,2}$ and $m^{k}_{2,3}$ respectively, defined as
\begin{align*}
	m^{k}_{2,1}(\xi)&=\widehat{w}_{\mathfrak{R},\lambda_k}(\xi)-\sum_{q=1}^{\infty}\sum_{a\in U_{q}}\sum_{b\in \mathbb{Z}_{q}^{n}}e_{q}(-\lambda_k a)F_{q}(a,b)\widehat{\zeta_{q}}\left(\xi-\frac{b}{q}\right)\widehat{d\sigma_{\mathfrak{R}}}\left(\lambda_k^{\frac{1}{d}}\left(\xi-\frac{b}{q}\right)\right),\\
	m^{k}_{2,2}(\xi)&=\Omega_{\lambda_k}(\xi)-\sum_{q=1}^{N}\sum_{a\in U_{q}}\sum_{b\in \mathbb{Z}_{q}^{n}}F_{q}(a,b)\widehat{\zeta_{q}}\left(\xi-\frac{b}{q}\right)\widehat{d\sigma_{\mathfrak{R}}}\left(\lambda_k^{\frac{1}{d}}\left(\xi-\frac{b}{q}\right)\right),\;\text{and}\\
	m^{k}_{2,3}(\xi)&=\sum_{q>N}\sum_{a\in U_{q}}\sum_{b\in \mathbb{Z}_{q}^{n}}e_{q}(-\lambda_k a)F_{q}(a,b)\widehat{\zeta_{q}}\left(\xi-\frac{b}{q}\right)\widehat{d\sigma_{\mathfrak{R}}}\left(\lambda_k^{\frac{1}{d}}\left(\xi-\frac{b}{q}.\right)\right)
\end{align*}
The rest of the proof is devoted to establishing the estimate \eqref{equation4.6} for terms $M_{1,1}$ and $M_{1,2}$ and the estimate \eqref{equation4.7} for terms $M_{2,1}$, $M_{2,2}$, and $M_{2,3}$.
\subsection{\textbf{Estimate for $M_{1,1}$}} Note that by the estimate for single averages from \Cref{Unique C-4-1} we have that
\[\langle M_{\lambda_k,\phi}^{\mathfrak{R}}f,g\rangle\lesssim |Q|\langle f\rangle_{Q,p}\langle g\rangle_{Q,q},\quad k\in\N.\] 
Therefore, 
\begin{align*}
    \langle M_{1,1}f,g\rangle&\leq \langle \sup\limits_{\lambda_{k}\leq \lambda_{[N^{\epsilon}]}} M_{\lambda_{k},\phi}^{\mathfrak{R}}f,g\rangle \\&\leq\sum_{k\leq [N^{\epsilon}]}\langle  M_{\lambda_{k},\phi}^{\mathfrak{R}}f,g\rangle \\
    &\leq N^{\epsilon} |Q|\langle f\rangle_{Q,p}\langle g\rangle_{Q,q}.
\end{align*}

\subsection{\textbf{Estimate for $M_{1,2}$}} We write the multiplier $\Omega_{\lambda_k}(\xi)=v_{\lambda_k}(\xi)s(\xi)$, where
\begin{align*}
	v_{\lambda_k}(\xi)&=\sum_{b\in\mathbb{Z}^{n}_{L}}\widehat{\zeta_{L}}\left(\xi-\frac{b}{L}\right)\widehat{d\sigma_{\mathfrak{R}}}\left({\lambda_k^{\frac{1}{d}}}\left(\xi-\frac{b}{L}\right)\right),\\
	\text{and}\quad s(\xi)&=\sum_{0\leq a<L}\sum_{b\in\mathbb{Z}^{n}_{L}}F_{L}(a,b)\widehat{\zeta_{2L}}\left(\xi-\frac{b}{L}\right).
\end{align*}
Let $V_{\lambda_k}$ and $S$ be the operators associated with the multipliers $v_{\lambda_k}$ and $s$ respectively. In order to prove estimates on the maximal operator corresponding to $v_{\lambda_k}$, we require a transference result from \cite{magyar}. In order to state this result we need the following setting. 

Let $K$ be a suitable distribution kernel and consider the convolution operator $Tf(x)=f* K(x)$. Denote $\widehat{K}(\xi)=m(\xi)$, thus $(Tf)^{\wedge}(\xi)=m(\xi)\widehat{f}(\xi)$. By restricting $K$ to $\Z^n$, i.e., $K_{dis}=K\vert_{\mathbb{Z}^{n}},$ consider the discrete operator $T_{dis}(f)=f\star K_{dis}$. Assume that $m(\xi)$ is supported in $\left[-\frac{1}{2q},\frac{1}{2q}\right]^{n},$ where $q$ is a fixed integer $q \geq 1$. Consider the periodic function
\[\tilde{m}^{q}(\xi)=\sum_{b\in\mathbb{Z}^{n}}m\left(\xi-\frac{b}{q}\right).\] 
We shall require Banach space valued transference result. Therefore, we shall assume that $B$ is a finite dimensional Banach space and $m$ is a $B$-valued bounded measurable function. With these notation, we recall the transference result from \cite{magyar}. 
\begin{proposition}\label{Unique C-3-9}	\cite{magyar}
	Let $T$ and $T_{dis}$ be the operators defined as above. Assume that $T$ is a bounded operator from $L^{p}(\mathbb{R}^{n})\rightarrow L^{p}_B(\mathbb{R}^{n})$.  Then\[\|T_{dis}^{q}\|_{\ell^p\rightarrow\ell^p_B}\lesssim \|T\|_{L^p\rightarrow L^p_B}.\] 
   \end{proposition}
Note that we shall use the result for maximal function defined by taking supremum with respect to the parameter $\lambda_k$. The transference result above is applied with Banach space $B$ of $\ell^{\infty}$ functions in the parameter $\lambda_k$. We shall take only finitely many $\lambda_k'$s and prove uniform estimates. A standard limiting argument can be used to deduce the corresponding results for the maximal function associated with the full sequence $\lambda_k$. 

Along with the transference result described above, we also require the following $\ell^p$-estimates for the operator $S$ in order to prove the desired estimate for the term $M_{1,2}$.
\begin{lemma}\label{Unique C-4-3}
	The operator $S$ is bounded on $\ell^{p}(\mathbb{Z}^{n})$ for $1\leq p\leq \infty.$
\end{lemma}
Let us first observe that by assuming the lemma above, we can conclude proof of the estimate for the term $M_{1,2}$. Indeed, for $\left(\frac{1}{p},\frac{1}{q}\right)\in V_n$, \Cref{Unique C-4-6} together with \Cref{Unique C-3-9} implies that 
\begin{align*}
	\langle M_{1,2}f,g\rangle&\leq\langle\sup\limits_{\lambda_{k}}|V_{\lambda_{k}}\circ Sf|,g\rangle\\
	&\lesssim|Q|\langle Sf\rangle_{Q,p}\langle g\rangle_{Q,q}\\
	&\lesssim|Q|\langle f\rangle_{Q,p}\langle g\rangle_{Q,q},
\end{align*}
	where we used \Cref{Unique C-4-3} in the last step.
\subsection{Proof of \Cref{Unique C-4-3}} We need the following estimate from \cite[Lemma 14]{cook1}.
	\begin{lemma}[\cite{cook1}]\label{cookexp}
		Let $L\in\N$ and $\mathcal{B}(\mathfrak{R})>2^d(d-1)$. Then,
		\[\frac{1}{L^{n-1}}\sum_{x\in\mathbb{Z}^{n}_{L}}1_{\{\mathfrak{R}(x)\equiv 0\text{ mod }L\}}(x)\lesssim 1.\]
	\end{lemma}
We shall prove the claimed $\ell^p$ estimate at $p=1$ and $p=\infty.$ The general case follows by standard interpolation argument. Note that it is enough to show $\|s^\vee\|_1\lesssim1$. Consider 
\begin{align*}
	s^{\vee}(-x)&=\sum_{0\leq a<L}\sum_{b\in \mathbb{Z}_{L}^{n}}F_{L}(a,b)\int_{\mathbb{T}^{n}}\widehat{\zeta_{L}}(\xi-\frac{b}{L})e^{i(-x)\cdot\xi}d\xi\\ 
	&=\sum_{0\leq a<L}\sum_{b\in \mathbb{Z}_{L}^{n}}F_{L}(a,b)\int_{\mathbb{T}^{n}}\widehat{\zeta_{L}}(\xi-\frac{b}{L})e^{-ix\cdot(\xi-\frac{b}{L})}e^{-x\frac{b}{L}}d\xi\\
	&=\zeta_{L}(x)\sum_{0\leq a<L}\frac{1}{L^{n}}\sum_{m\in \mathbb{Z}_{L}^{n}}\sum_{b\in \mathbb{Z}_{L}^{n}}e_{L}(a\mathfrak{R}(m)+(m-x)\cdot b)\\
	&=\zeta_{L}(x)\frac{1}{L^{n-1}}\sum_{m\in \mathbb{Z}_{L}^{n}}\sum_{b\in \mathbb{Z}_{L}^{n}}e_{L}((m-x)\cdot b)1_{\{\mathfrak{R}(m)\equiv 0\text{ mod }L\}}(x)\\
	&=\frac{\zeta_{L}(x)}{L^{n-1}}\sum_{b\in \mathbb{Z}_{L}^{n}}1_{\{\mathfrak{R}(x)\equiv 0\text{ mod }L\}}(x)\\
	&=L\zeta_{L}(x)1_{\{\mathfrak{R}(x)\equiv 0\text{ mod }L\}}(x)\\
	&=\frac{\zeta(x/L)}{L^{n-1}}1_{\{\mathfrak{R}(x)\equiv 0\text{ mod }L\}}(x),
\end{align*}
Since $\zeta$ is a Schwartz class function, we have that
\begin{align*}
	\frac{1}{L^{n-1}}|\zeta(x/L)1_{\{\mathfrak{R}(x)\equiv 0\text{ mod }L\}}|&\lesssim  \frac{1}{L^{n-1}}\frac{1}{\left(\frac{|x|}{L}\right)^{n+1}}\lesssim \frac{L^2}{\|x\|_{\ell^{2}}^{n+1}}.
\end{align*} 
%Therefore, we get that 
%\[\frac{1}{L^{n-1}}|\zeta(x/L)1_{\{\mathfrak{R}(x)\equiv 0\text{ mod }L\}}|\lesssim  \frac{1}{L^{n-1}}\frac{L^{n+1}}{(|x|)^{n+1}}\lesssim \frac{L^2}{\|x\|_{\ell^{2}}^{n+1}}.\]
Therefore,
\[\frac{1}{L^{n-1}}\sum_{\|x\|_{\ell^{\infty}}\geq L}\zeta(x/L)1_{\{\mathfrak{R}(x)\equiv 0\text{ mod }L\}}\lesssim L^{2}\sum_{\|x\|_{\ell^2}\geq L}\frac{1}{\|x\|_{\ell^2}^{n+1}}\lesssim1.\]
This implies that 
\begin{align*}
	\|s^\vee \|_{\ell^{1}}&\leq\sum_{x\in\mathbb{Z}^{n}}|s^{\vee}(x)|\\
	&\leq\sum_{x\in\mathbb{Z}^{n}}\frac{|\zeta(x/L)|}{L^{n-1}}1_{\{\mathfrak{R}(x)\equiv 0\text{ mod }L\}}(x)\\
	&\lesssim\Big(\frac{1}{L^{n-1}}\sum_{x\in\mathbb{Z}^{n}_{L}}1_{\{\mathfrak{R}(x)\equiv 0\text{ mod }L\}}(x)+\frac{1}{L^{n-1}}\sum_{\|x\|_{\ell^{\infty}}\geq L}|\zeta(x/L)|1_{\{\mathfrak{R}(x)\equiv 0\text{ mod }L\}}(x)\Big).
\end{align*}
Invoking \Cref{cookexp} for the estimate of the first term in the above along with the estimate we get the desired result. This concludes the proof of \Cref{Unique C-4-3}.
\end{proof}
\subsection{\textbf{Estimate for $M_{2,1}$}}
From the Fourier transform estimate \eqref{equation4.3} we obtain that
\[\|M^{k}_{2,1}f\|_{\ell^{2}}\lesssim \lambda_k^{-\eta_{\mathfrak{R}}}\|f\|_{\ell^2},\]
and hence 
\begin{align*}
	\|M_{2,1}f\|_{\ell^{2}}&\lesssim \Big(\sum_{\lambda_{k}>\lambda_{[N^\epsilon]}}\lambda_k^{-\eta_{\mathfrak{R}}}\Big)\|f\|_{\ell^2}\\
	&\lesssim\lambda_{N^\epsilon}^{-\eta_{\mathfrak{R}}}\|f\|_{\ell^2} \lesssim N^{-(c_{\mathfrak{R}}-2)}\|f\|_{\ell^{2}}.
\end{align*}
\subsection{\textbf{Estimate for $M_{2,2}$}}	
We have,
\begin{align*}
	m^{k}_{2,2}(\xi)&=\Omega_{\lambda_{k}}(\xi)-\sum_{q=1}^{N}\sum_{a\in U_{q}}\sum_{b\in \mathbb{Z}_{q}^{n}}F_{q}(a,b)\widehat{\zeta_{q}}\left(\xi-\frac{b}{q}\right)\widehat{d\sigma_{\mathfrak{R}}}\left(\lambda_{k}^{\frac{1}{d}}\left(\xi-\frac{b}{q}\right)\right)\\
	&=\sum_{0\leq \bar{a}<L}\sum_{\bar{b}\in \mathbb{Z}^{n}_{L}}F_{L}(\bar{a},\bar{b})\widehat{\zeta_{2L}}\left(\xi-\frac{b}{L}\right)\widehat{d\sigma_{\mathfrak{R}}}\left(\lambda_{k}^{\frac{1}{d}}\left(\xi-\frac{b}{L}\right)\right)\\&\hspace{1.5cm} -\sum_{q=1}^{N}\sum_{a\in U_{q}}\sum_{b\in \mathbb{Z}_{q}^{n}}F_{q}(a,b)\widehat{\zeta_{q}}\left(\xi-\frac{b}{q}\right)\widehat{d\sigma_{\mathfrak{R}}}\left(\lambda_{k}^{\frac{1}{d}}\left(\xi-\frac{b}{q}\right)\right).
\end{align*}
We note that for any function $g:\left\{\frac{a}{L}: 0\leq a<L\right\}\to\R$ we have the identity
\[\sum_{a\in\mathbb{Z}_{L}}g\left(\frac{a}{L}\right)=\sum_{1\leq q\leq N}\sum_{a\in U_{q}}g\left(\frac{a}{q}\right)+\sum_{\substack{q:q|L\\q>N}}\sum_{a\in U_{q}}g\left(\frac{a}{q}\right).\]
Using the identity as above, we write the multiplier as $m^{k}_{2,2}=m^{k}_{2,2,1}+m^{k}_{2,2,2}$, where
\begin{align*}
	m^{k}_{2,2,1}(\xi)&=\sum_{0\leq \bar{a}<L}\sum_{\substack{\bar{b}\in \mathbb{Z}^{n}_{L}\\ \frac{L}{\nu} \leq N}}F_{L}(\bar{a},\bar{b})\left( \widehat{\zeta_{2L}}\left(\xi-\frac{\bar{b}}{L}\right)-\widehat{\zeta_{\frac{L}{\nu}}}\left(\xi-\frac{\bar{b}}{L}\right) \right)\widehat{d\sigma_{\mathfrak{R}}}\left(\lambda_{k}^{\frac{1}{d}}\left(\xi-\frac{\bar{b}}{L}\right)\right),\\
	\text{and}\quad m^{k}_{2,2,2}(\xi)&=\sum_{0\leq \bar{a}<L}\sum_{\substack{\bar{b}\in \mathbb{Z}^{n}_{L}\\ \frac{L}{\nu}> N}}F_{L}(\bar{a},\bar{b})\widehat{\zeta_{2L}}\left(\xi-\frac{\bar{b}}{L}\right)\widehat{d\sigma_{\mathfrak{R}}}\left(\lambda_{k}^{\frac{1}{d}}\left(\xi-\frac{\bar{b}}{L}\right)\right),
\end{align*}
with the index $\nu=\text{gcd}(\bar{a}, \bar{b}, L)$ for which we have $F_{L}(\bar{a}, \bar{b})=F_{\frac{L}{\nu}}(\frac{\bar{a}}{\nu}, \frac{\bar{b}}{\nu})$. Thus we have
\[M_{2,2}f(x)\leq M_{2,2,1}f(x)+M_{2,2,2}f(x),\]
where $M_{2,2,1}$ and $M_{2,2,2}$ are the operators defined as
\[M_{2,2,j}f(x)=\sup_{\lambda_k>\lambda_{[N^\epsilon]}}|(m^k_{2,2,j})^\vee\star f(x)|,\;j=1,2.\]
\subsection{Estimate for $M_{2,2,1}$:} Observe that  $\widehat{\zeta_{2L}}(\xi-\frac{\bar{b}}{L})-\widehat{\zeta_{\frac{L}{\nu}}}(\xi-\frac{\bar{b}}{L})$ is supported for $|2L(\xi-\frac{\bar{b}}{L})|\geq \frac{1}{2}$; i.e., $|\xi-\frac{\bar{b}}{L}|\geq \frac{1}{4L}$. Therefore, in the support of $\widehat{\zeta_{2L}}(\xi-\frac{\bar{b}}{L})-\widehat{\zeta_{\frac{L}{\nu}}}(\xi-\frac{\bar{b}}{L})$ the term  $|\widehat{d\sigma_{\mathfrak{R}}}(\lambda^{\frac{1}{d}}_{k}(\xi-\frac{\bar{b}}{L}))|$ is bounded by $\left(\frac{L}{\lambda_k^{\frac{1}{d}}}\right)^{c_{\mathfrak{R}}-1}$. Also, for $k>N_{\circ}$, $\lambda_k>N^{d}!>(N!)^{d}=L^{d}$.
\begin{align*}
	\|M_{2,2,1}f\|_{\ell^{2}}&\lesssim\sum_{\lambda_{k}>\lambda_{N^{\epsilon}}}\left(\sup_{\xi}|m^{k}_{2,2,1}(\xi)|\right)\|f\|_{\ell^{2}}\\
	&\lesssim\|f\|_{\ell^{2}}\sum_{0\leq \bar{a}<L}\sum_{\lambda_{k}>\lambda_{N^{\epsilon}}}\left(\frac{L}{\lambda_k^{\frac{1}{d}}}\right)^{c_{\mathfrak{R}}-1}\frac{1}{L^{c_{\mathfrak{R}}}}\\
	&\lesssim\|f\|_{\ell^{2}}\sum_{0\leq \bar{a}<L}\frac{1}{L(\lambda_{N^{\epsilon}}^{\frac{1}{d}})^{c_{\mathfrak{R}}-2}}\\
	&\lesssim\frac{1}{L^{c_{\mathfrak{R}}-2}}\|f\|_{\ell^{2}}.
\end{align*}

\subsection{Estimate for $M_{2,2,2}$:} Consider
\begin{align*}
	m^{k}_{2,2,2}(\xi)&=\sum_{0\leq \bar{a}<L}\sum_{\substack{\bar{b}\in \mathbb{Z}^{n}_{L}\\ \frac{L}{\nu}> N}}F_{L}(\bar{a},\bar{b})\widehat{\zeta_{2L}}\left(\xi-\frac{\bar{b}}{L}\right)\widehat{d\sigma_{\mathfrak{R}}}\left(\lambda_{k}^{\frac{1}{d}}\left(\xi-\frac{\bar{b}}{L}\right)\right)\\&=\sum_{0\leq \bar{a}<L}\sum_{\substack{\bar{b}\in \mathbb{Z}^{n}_{L}\\ \frac{L}{\nu}> N}}F_{L}(\bar{a},\bar{b})\widehat{\zeta_{2L}}\left(\xi-\frac{\bar{b}}{L}\right)\widehat{\zeta_{L}}\left(\xi-\frac{\bar{b}}{L}\right)\widehat{d\sigma_{\mathfrak{R}}}\left(\lambda_{k}^{\frac{1}{d}}\left(\xi-\frac{\bar{b}}{L}\right)\right).
\end{align*}
Since $\widehat{\zeta_{L}}\left(\xi-\frac{b_1}{L}\right)$ and $\widehat{\zeta_{2L}}\left(\xi-\frac{b_2}{L}\right)$ have disjoint supports for $b_1\neq b_2$, we can write 
\[M^{k}_{2,2,2}=V_{\lambda_k}\circ S,\]
where $V_{\lambda_k}$ and $S$ are the operators corresponding to the multipliers $v_{\lambda_k}$ and $s$ respectively given by
\begin{align*}
	s(\xi)&=\sum_{0\leq \bar{a}<L}\sum_{\substack{\bar{b}\in \mathbb{Z}^{n}_{L}\\ \frac{L}{\nu}> N}}F_{L}(\bar{a},\bar{b})\widehat{\zeta_{2L}}\left(\xi-\frac{\bar{b}}{L}\right)\quad \text{and}\\
	v_{\lambda_k}(\xi)&=\sum_{\substack{\bar{b}\in \mathbb{Z}^{n}_{L}\\ \frac{L}{\nu}> N}}\widehat{\zeta_{L}}\left(\xi-\frac{\bar{b}}{L}\right)\widehat{d\sigma_{\mathfrak{R}}}\left({\lambda^{\frac{1}{d}}_k}\left(\xi-\frac{\bar{b}}{L}\right)\right).
\end{align*}  
Invoking \Cref{Unique C-3-9} and \Cref{Unique C-4-6}, we get that $\|\sup\limits_{\lambda_{k}}|V_{\lambda_k}f|\|_{\ell^{p}}\lesssim \|f\|_{\ell^{p}}$ for $p>1$; in particular, note that the  inequality as above holds true for $p=2$. This gives us 
\begin{align*}
		\|M_{2,2,2}f\|_{\ell^{2}}&\lesssim\|Sf\|_{\ell^2}\\
		&\lesssim\sup\limits_{\xi}|s(\xi)| \|f\|_{\ell^{2}}\\
		&\lesssim\sum_{0\leq\bar{a}<L}|F_L(\bar{a},\bar{b})|\|f\|_{\ell^2}\\&\lesssim\frac{1}{L^{c_{\mathfrak{R}}-1}} \|f\|_{\ell^{2}}.
		\end{align*}
\subsection{\textbf{Estimate for $M_{2,3}$}}
		Since $e_{q}(-\lambda_{k} a)=1$ for $\lambda_k>\lambda_{[N^\epsilon]}$ and $q>N$, we write the multiplier $m^{k}_{2,3}$ as 
\begin{align*}
	m^{k}_{2,3}(\xi)&=\sum_{q>N}\sum_{a\in U_{q}}\sum_{b\in 			\mathbb{Z}_{q}^{n}}F_{q}(a,b)\widehat{\zeta_{q}}\left(\xi-			\frac{b}{q}\right)\widehat{d\sigma_{\mathfrak{R}}}					\left(\lambda_{k}^{\frac{1}{d}}\left(\xi-\frac{b}{q}\right)			\right),\\
	&=\left(\sum\limits_{b\in \mathbb{Z}_{q}^{n}}						\widehat{\zeta_{q}}(\xi-\frac{b}{q})								\widehat{d\sigma_{\mathfrak{R}}}(\lambda_{k}^{\frac{1}{d}}			(\xi-\frac{b}{q}))\right)\left(\sum\limits_{a\in U_{q}}				\sum\limits_{b\in \mathbb{Z}_{q}^{n}}F_{q}(a,b)						\widehat{\zeta_{2q}}(\xi-\frac{b}{q})\right).
\end{align*}
		Using \Cref{Unique C-3-9} and \Cref{Unique C-4-6}, we get that the maximal operator corresponding to the first term in the expression above is bounded on $\ell^2$. Hence,
\begin{align*}
	\|M_{2,3}f\|_{\ell^{2}}&\lesssim \|f\|_{\ell^{2}}\sum_{q>N}			\sum_{a\in U_{q}}|F_{q}(a,b)|\\&\lesssim \|f\|_{\ell^{2}}			\sum_{q>N}\frac{1}{q^{c_{\mathfrak{R}}-1}}\\&\lesssim 				\frac{1}{N^{c_{\mathfrak{R}}-1}}\|f\|_{\ell^{2}}.
\end{align*}
	This concludes the proof of \Cref{Unique C-1-9}.

\begin{proof}[\textbf{Proof of \Cref{Unique C-1-8}:}]
	The proof is similar to that of \Cref{Unique C-1-9} with the exception that we use the following bound for the single scale average instead of \Cref{Unique C-4-1}. For any $\delta>0$, $\left(\frac{1}{p}, \frac{1}{q}\right) \in P_{n}$ and any cube $Q\subset\Z^n$, there exists $C>0$, independent of $\lambda$, such that
	\[\langle M_{\lambda}f,g\rangle\leq C\lambda^{\delta}\langle f\rangle_{Q,p}\langle g\rangle_{Q,q}|Q|,\;\forall\;f,g\;\text{supported in } Q.\]
	The estimate above is a direct consequence of the following $\ell^p-$improving property of $M_\lambda$ obtained in \cite{Hughes},
	\begin{lemma}[\cite{Hughes}]
	Let $n\geq 5$ and $\frac{n+1}{n-1}\leq p\leq 2$. Then for each $\delta>0$ there exists a constant $C_{p,\delta}$ such that for all $\lambda\in\N$, we have the following estimate 
	\[\|M_{\lambda}f\|_{\ell^{p'}}\lesssim_{p,\delta} \lambda^{\delta-n\left(\frac{2}{p}-1\right)}\|f\|_{\ell^p}.\]
	\end{lemma}		
\end{proof}
\appendix
	\section{Sparse bounds for a continuous maximal operator}\label{Appendix}
	In this section we discuss the continuous analogue of the lacunary maximal function associated with the average $M_{\lambda_{k},\phi}^{\mathfrak{R}}$. As mentioned previously, due to non-availability of results in the desired framework in the literature; we present the proofs here. Let $\{\lambda_k\}_{k=1}^{\infty}$ be a lacunary sequence in the set of regular values $\Lambda$. Consider the lacunary maximal function defined by
	\[\mathcal{A}_{\mathfrak{R},*}f(x)=\sup\limits_{\lambda_k}\mathcal{A}_{\lambda_{k},\phi}^{\mathfrak{R}}f(x), \]
	where 
	\begin{equation*}
		\mathcal{A}_{\lambda_{k},\phi}^{\mathfrak{R}}f(x):=\frac{1}{\lambda_{k}^{\frac{n}{d}-1}}f*d\sigma_{\mathfrak{R},\lambda_k}(x)=f*d\bar{\sigma}_{\mathfrak{R},\lambda_k}(x),
	\end{equation*}
	where  $\phi$ and $\mathfrak{R}$ are as in \Cref{sec:intro}. \\
	We have the following sparse bound for the operator $\mathcal{A}_{\mathfrak{R},*}$, which were used to prove \Cref{Unique C-1-9}. 
	\begin{theorem}\label{Unique C-4-6}
		The sparse bound $\|\mathcal{A}_{\mathfrak{R},*}\|_{p,q}$ holds for all points $\left(\frac{1}{p},\frac{1}{q}\right)$ belonging to the interior of the triangle $V_n$ with vertices $(1,0)$, $\left(\frac{2c_{\mathfrak{R}}-1}{2c_{\mathfrak{R}}}, \frac{2c_{\mathfrak{R}}-1}{2c_{\mathfrak{R}}}\right)$, and $(0,1)$.
	\end{theorem}
	The sparse bounds as above are indeed a consequence of the $L^p$-improving properties of single scale avarages . This observation was obtained in the work of \cite{Lacey2} and further used by Oberlin \cite{Oberlin} to obtain certain endpoint sparse bounds for Radon transforms. We also refer the readers to \cite{BRS}. We use the following result
	 \begin{lemma}\label{Unique C-4-7}\cite[Theorem 1.3]{Oberlin}
		 Suppose $\sigma$ is a finite measure supported on the unit ball with 
		\[|\widehat{\sigma}(\xi)|\lesssim |\xi|^{-\beta},\]
		for some $\beta>0$, and $1<p<q<\infty$ are exponents such that convolution with $\sigma$ is a bounded operator from $L^p$ to $L^q$. Then for every pair of compactly supported $f_1$, $f_2$ there is a sparse collection of cubes $\mathcal{Q}$ such that 
		 \[|\langle T^*f_1,f_2\rangle|\lesssim \sum\limits_{Q\in\mathcal{Q}}|Q|\langle f_1\rangle_{Q,p}\langle f_2\rangle_{3Q,q^{'}},\] with $T^*f(x)=\sup\limits_{j}\sigma_j*|f|(x)$, where $\sigma_j$ is defined as \[\int fd\sigma_j=\int f(2^{j}x)d\sigma(x).\]
	 \end{lemma}
	 In view of the lemma above, \Cref{Unique C-4-6} follows from the following estimates for single scale averages.
	 \begin{proposition}\label{Unique C-4-8}
	    The operator $\mathcal{A}_{1,\phi}^{\mathfrak{R}}$ is bounded from $L^p(\R^n)$ to $L^q(\R^n)$ for  $\left(\frac{1}{p},\frac{1}{q}\right)$ belonging to the triangle $D_n$ with vertices $(0,0)$, $\left(\frac{2c_{\mathfrak{R}}-1}{2c_{\mathfrak{R}}}, \frac{1}{2c_{\mathfrak{R}}}\right)$, and $(1,1)$.
	\end{proposition}
	   Note that, the previously defined triangle $S_n$ is strictly contained inside this triangle $V_n$. Proof of \Cref{Unique C-4-8} uses Bourgain's interpolation trick. We state Lemma 2.6 from \cite{Sanghyuk} for convenience.
	 \begin{lemma}[\cite{Sanghyuk}]\label{Unique C-4-9}
		 Let $\epsilon_1,\epsilon_2>0$. Suppose that $\{T_j\}$ is a sequence of linear (or sublinear) operators 
		 such that for some $1\leq p_1,p_2<\infty$, and $1\leq q_1,q_2<\infty$, 
		 $$\Vert T_{j}(f)\Vert_{L^{q_1}}\leq M_12^{\epsilon_1 j}\Vert f\Vert_{L^{p_1}},~~\Vert T_{j}(f)\Vert_{L^{q_2}}\leq M_22^{-\epsilon_2 j}\Vert f\Vert_{L^{p_2}}.$$
		 Then $T=\sum_jT_j$ is bounded from $L^{p,1}$ to $L^{q,\infty}$, i.e. 
		 $$\Vert T(f)\Vert_{L^{q,\infty}}\lesssim M^{\theta}_{1}M^{1-\theta}_{2}\Vert f\Vert_{L^{p,1}},$$	
		 where $\theta=\frac{\epsilon_2}{\epsilon_1+\epsilon_2}$, $\frac{1}{q}=\frac{\theta}{q_1}+\frac{1-\theta}{q_2}$ 
		 and $\frac{1}{p}=\frac{\theta}{p_1}+\frac{1-\theta}{p_2}$.
	 \end{lemma}
	\begin{proof}[\textbf{Proof of \Cref{Unique C-4-8}}]
	Since $\|d\bar{\sigma}_{\mathfrak{R},1}\|_1\lesssim1$, it follows that the operator $\mathcal{A}_{1,\phi}^{\mathfrak{R}}$ is bounded in $L^p$ for $p=1$ and $p=\infty$. Next, we prove that $\mathcal{A}_{1,\phi}^{\mathfrak{R}}$ is bounded from $L^{\frac{2c_{\mathfrak{R}}}{2c_{\mathfrak{R}}-1},1}$ to $L^{2c_{\mathfrak{R}},\infty}$.  Note that the desired $L^p\rightarrow L^q$-boundedness of $\mathcal{A}_{1,\phi}^{\mathfrak{R}}$ in the interior of the triangle $D_n$ follows by interpolation. \\
	
	Let $\psi$ be a Schwartz function such that 
		$$\widehat{\psi}(\xi)=
		 \begin{cases}
			 1 , & \text{if} \hspace{2mm} |\xi| \leq 1, \\
	
			 0 , & \text{if} \hspace{2mm} |\xi| \geq 2.
		 \end{cases}$$ Define $\widehat{\psi_{N^{-1}}}(\xi)=\widehat{\psi}\left(\frac{\xi}{N}\right)$ and write  $\widehat{f}(\xi)\widehat{d\bar{\sigma}_{\mathfrak{R},1}}(\xi)$ as \[\widehat{f}(\xi)\widehat{d\bar{\sigma}_{\mathfrak{R},1}}(\xi)=\widehat{f}(\xi)\widehat{\psi_{N^{-1}}}(\xi)\widehat{d\bar\sigma_{\mathfrak{R},1}}(\xi)+\widehat{f}(\xi)(1-\widehat{\psi_{N^{-1}}}(\xi))\widehat{d\bar\sigma_{\mathfrak{R},1}}(\xi).\]
		 Notice that, $\psi_{N^{-1}}*{d\bar\sigma}_{\mathfrak{R},1}(x)$ is bounded by $\frac{N}{(1+|x|)^M}$. If $\mathfrak{R}(x):=\sum\limits_{i=1}^{n}|x_i|^2$,  a proof of this can be found in \cite[Lemma 6.5.3]{Grafakos1}. Further, since the proof depends on the dimensionality of the measure, this estimate is valid for the measure $\bar{d\sigma}_{\mathfrak{R},1}$ over the variety $\{x\in\R^n:\mathfrak{R}(x)=1\}$, where $\mathfrak{R}$ is any integral form of degree $d$. We refer to \cite[p 970]{Hughes}, where this estimate is used to bound the averages.
		 Consider, 
		 \begin{align*}
	\|f*\psi_{N^{-1}}*{d\bar\sigma}_{\mathfrak{R},1}\|_{L^\infty}&\lesssim\|\psi_{N^{-1}}*{d\bar\sigma}_{\mathfrak{R},1}\|_{L^\infty}\|f\|_{L^1}\\&\lesssim N\|f\|_{L^1}.
		 \end{align*}
		Next, using the support of $\widehat{\psi_{N^{-1}}}(\xi)$ and the bound of $\widehat{d\bar\sigma_{\mathfrak{R},1}}(\xi)$ mentioned earlier, we get that 
		\begin{align*}
			\|f*(1-\psi_{N^{-1}})*d\bar\sigma_{\mathfrak{R},1}\|		_{L^2}&=\|\widehat{f}(\xi)(1-\widehat{\psi_{N^{-1}}}(\xi))\widehat{d\bar\sigma_{\mathfrak{R},1}}(\xi)\|_{L^2}\\
			&\lesssim\sup\limits_{\xi}|(1-\widehat{\psi_{N^{-1}}}(\xi))\widehat{d\bar\sigma_{\mathfrak{R},1}}(\xi)|\|f\|_{L^2}\\
			&\lesssim\frac{1}{(1+|\xi|)^{c_{\mathfrak{R}}-1}}|1-\widehat{\psi_{N^{-1}}}(\xi)|\|f\|_{L^2}\\
			&\lesssim N^{1-c_{\mathfrak{R}}}\|f\|_{L^2}.
		\end{align*} 
		Using \Cref{Unique C-4-9}, it follows that
		\[\|\mathcal{A}_{1,\phi}^{\mathfrak{R}}f\|_{L^{2c_{\mathfrak{R}},\infty}}\lesssim\|f\|_{L^{\frac{2c_{\mathfrak{R}}}{2c_{\mathfrak{R}}-1},1}}.\]
		This completes the proof.
	\end{proof}	
	\subsection*{Acknowledgement}
	Ankit Bhojak is supported by the Science and Engineering Research Board, Department of Science and Technology, Govt. of India, under the scheme National Post-Doctoral Fellowship, file no. PDF/2023/000708. Surjeet Singh Choudhary is supported by IISER Mohali for postdoctoral fellowship. Siddhartha Samanta is supported by IISER Bhopal for PhD fellowship. Saurabh Shrivastava acknowledges the financial support from Science and Engineering Research Board, Govt. of India, under the scheme Core Research Grant, file no. CRG/2021/000230.
	\bibliography{bibliography}
\end{document}